\documentclass[11pt]{amsart}
\usepackage{geometry}                
\usepackage{graphicx}
\usepackage{amsfonts,amsthm,amsmath,amssymb,latexsym, amscd, euscript}
\usepackage{graphicx,color}
\usepackage[all]{xy}
\usepackage{epsfig}
\usepackage{labelfig}
\usepackage{mathrsfs}

\usepackage{amssymb}
\usepackage{epstopdf}
\begin{document}

\author{Dragomir \v Sari\' c}

\address{Mathematics Ph.D. Program,
CUNY Graduate Center, 365 Fifth Avenue, Room 4208, New York, NY
10016-4309 and, \vskip .1 cm Department of Mathematics, Queens
College, The City University of New York, 237 Kiely Hall, 65-30
Kissena Blvd., Flushing, NY 11367}

\email{Dragomir.Saric@qc.cuny.edu}

\theoremstyle{definition} 

 \newtheorem{definition}{Definition}[section]
 \newtheorem{remark}[definition]{Remark}
 \newtheorem{example}[definition]{Example}

\newtheorem*{notation}{Notation}  

\theoremstyle{plain}      

 \newtheorem{proposition}[definition]{Proposition}
 \newtheorem{theorem}[definition]{Theorem}
 \newtheorem{corollary}[definition]{Corollary}
 \newtheorem{lemma}[definition]{Lemma}

\def\H{{\mathbb H}}
\def\F{{\mathcal F}}
\def\R{{\mathbb R}}
\def\Q{{\mathbb Q}}
\def\Z{{\mathbb Z}}
\def\E{{\mathcal E}}
\def\N{{\mathbb N}}
\def\X{{\mathcal X}}
\def\Y{{\mathcal Y}}
\def\C{{\mathbf C}}
\def\D{{\mathbf D}}
\def\G{{\mathcal G}}
\def\T{{\mathcal T}}

\title{Shears for quasisymmetric maps}


\date{\today}

\maketitle

\begin{abstract}
We give an elementary proof of a theorem that characterizes quasisymmetric maps of the unit circle in terms of shear coordinates on the Farey tesselation. The proof only uses the normal family argument for quasisymmetric maps and some elementary hyperbolic geometry. 
\end{abstract}

\section{Introduction}

The {\it universal Teichm\"uller space} $\mathcal{T}$ is the space of all quasisymmetric maps of the unit circle $S^1$ that fix $1$, $i$ and $-1$. The Teichm\"uller space of any Riemann surface embeds as a closed complex submanifold in the universal Teichm\"uller space $\mathcal{T}$ (see \cite{Bers}, \cite{GL}). In addition, $\mathcal{T}$ contains various subspaces defined by the smoothness properties of the circle maps that are of interest in analysis and mathematical physics (see \cite{NagSullivan}, \cite{TakhTeo}, \cite{Bishop}, \cite{Shen} \cite{Astala}). 

The unit disk $\mathbb{D}=\{z:|z|<1\}$ is equipped with the hyperbolic metric $\rho (z)|dz|=\frac{2|dz|}{1-|z|^2}$ and all geometric considerations are with respect to this metric. The unit circle $S^1$ is the visual boundary of $\mathbb{D}$. We parametrize the universal Teichm\"uller space $\mathcal{T}$ in terms of the hyperbolic invariants on $\mathbb{D}$.

 Let $\Delta_0\subset\mathbb{D}$ be an ideal geodesic triangle with vertices $1$, $i$ and $-1$ on $S^1$. The Farey tesselation $\mathcal{F}$ of $\mathbb{D}$ is a triangulation by ideal geodesic triangles which is obtained by repeatedly taking the hyperbolic reflections of $\Delta_0$ in its sides (for example, see \cite{Penner}, \cite{Bonahon}). We implicitly identify $\mathcal{F}$ with the set of its edges.

A geodesic in $\mathbb{D}$ is identified with the pair of its ideal endpoints on the unit circle $S^1$. 
Using this identification, an orientation preserving homeomorphism $h:S^1\to S^1$ that fixes $1$, $i$ and $-1$ maps the edges of the Farey tesselation $\mathcal{F}$ onto edges of another ideal geodesic triangulation $h(\mathcal{F})$ of the unit disk $\mathbb{D}$. The triangulation $h(\mathcal{F})$ has $\Delta_0$ as a triangle since $h$ fixes $1$, $i$ and $-1$. For any two adjacent triangles $\Delta_1$ and $\Delta_2$ of $\mathcal{F}$ with common boundary side $f\in\mathcal{F}$, the hyperbolic inversion in $f$ interchanges them. However, this is not the case for the adjacent triangles in $h(\mathcal{F})$. If $h(\Delta_1)$ is given then the position of $h(\Delta_2)$ is determined by a single real number. Indeed, let $\Delta_2'$ be the image of $h(\Delta_1)$ under the hyperbolic inversion in the boundary side $h(f)$. Then $h(\Delta_2)$ is the image of $\Delta_2'$ under a unique hyperbolic translation with the translation axis $h(f)$ oriented as a boundary side of $h(\Delta_1)$. The translation length is a real number called the {\it shear} of $h(\Delta_1)$ and $h(\Delta_2)$. Thus a homeomorphism $h:S^1\to S^1$ induces a shear function $s_h:\mathcal{F}\to\mathbb{R}$ which assign to each edge $f\in \mathcal{F}$ the shear of the image of the two adjacent triangles (see \cite{Penner}).

Let $s:\mathcal{F}\to\mathbb{R}$ be an arbitrary ``shear'' function. 
Then there exists a developing map $h_s$ from the vertices of $\mathcal{F}$ into $S^1$ which realizes $s$. The vertices of $\mathcal{F}$ are dense in $S^1$ and yet $h_s$ does not always extend to a homeomorphism (see \cite{Penner}). 
A question of Penner \cite{Penner} is to characterize which shear functions give rise to homeomorphisms, quasisymmetric maps, symmetric maps and other smoothness classes. In \cite{Saric1}, a characterization for homeomorphisms and symmetric maps was established, and a condition for quasisymmetric maps is given. We give an elementary proof for the characterization of quasisymmetric maps in terms of shear functions.


Let $\Delta_1$ and $\Delta_2$ be two ideal geodesic triangles with disjoint interiors and a common boundary geodesic $f\in\mathcal{F}$. Let $z\in S^1$ be one ideal endpoint of $h(f)$ and let $C$ be an arbitrary horocycle based at $h(z)$. Denote by $\delta_i$ the length of the arc of $C$ that is inside the triangle $h(\Delta_i)$ for $i=1,2$. Note that the quotient $\delta_1/\delta_2$ is independent of the choice of a horocycle based at $h(z)$ because the restriction  of the hyperbolic metric on one horocycle scales to the restriction on the other horocycle. A {\it shear} on the common geodesic $f$ induced by the pair of ideal geodesic triangles $h(\Delta_1)$ and $h(\Delta_2)$ is given by
\begin{equation}
\label{eq:def_shear}
s(f)=\log \frac{\delta_1}{\delta_2},
\end{equation}
where $h(\Delta_2)$ comes before $h(\Delta_1)$ for the orientation of $C$. The above definition of the shear is independent of the choice of the endpoint of $f$ and the choice of a horocycle at the endpoint, and it is equivalent to the definition using the translation length (see \cite{Penner}).

Let $p\in S^1$ be an ideal endpoint of an edge of $\mathcal{F}$. A {\it fan of edges} $\mathcal{F}^p$ with the {\it tip} $p$ consists of all edges of $\mathcal{F}$ with one endpoint $p$. 
Fix a horocycle $C_p$ based at $p$ and index $\mathcal{F}^p$ by $\{ f_k^p\}_{k\in\mathbb{Z}}$ such that the point $f_k^p\cap C_p$ comes before $f_{k+1}^p\cap C_p$ for the natural orientation of $C_p$ as a boundary of the horoball and for all $k\in\mathbb{Z}$. Let $\delta_k^p$ be the length of the arc of  $C_p$ that is between the edges $f_k^p$ and $f_{k+1}^p$.

We give a new proof of the following theorem (see \cite{Saric1}, \S 5, and \cite{FanHu}).

\begin{theorem} 
\label{thm:main} A function $s:\mathcal{F}\to \R$ is
induced by shears of $h(\mathcal{F})$ of a quasisymmetric map of $h:S^1\to S^1$ if and only if there exists a constant
$M\geq 1$ such that for each fan of geodesics
$\mathcal{F}^p$ of $\F$ and for all $m\in\Z$ and $k\in\mathbb{N}\cup\{ 0\}$, we have
\begin{equation}
\label{eq:qs_horocycle}
\frac{1}{M}\leq\frac{\delta_m^p+\delta_{m+1}^p+\cdots +\delta_{m+k}^p}{\delta_{m-1}^p+\delta_{m-2}^p+\cdots +\delta_{m-k-1}^p}\leq M,
\end{equation}
where $\delta_n^p$ is the length of the arc of $C_p$ between $f_n^p$ and $f_{n+1}^p$.
\end{theorem}

By (\ref{eq:def_shear}), the inequality (\ref{eq:qs_horocycle}) is equivalent to
\begin{equation}
\label{eq:qs_shears}
\frac{1}{M}\leq e^{s(f_m^p)}\frac{1+e^{s(f_{m+1}^p)}+\cdots
+e^{s(f_{m+1}^p)+s(f_{m+2}^p)+\cdots
+s(f_{m+k}^p)}}{1+e^{-s(f_{m-1}^p)}+\cdots
+e^{-s(f_{m-1}^p)-s(f_{m-2}^p)-\cdots -s(f_{m-k}^p)}}\leq M.
\end{equation}
The inequalities (\ref{eq:qs_horocycle}) and (\ref{eq:qs_shears}) imply that $\delta_n^p$ are  between two positive constants and that $s(f_n^p)$ are  between two real constants for all tips $p$ and integers $n$. In addition, Theorem \ref{thm:main} implies that if $s:\mathcal{F}\to\mathbb{R}$ satisfies (\ref{eq:qs_shears}) then $h:S^1\to S^1$ is a homeomorphism.

\vskip .2 cm

\noindent
{\it Acknowledgements.} We wish to thank an anonymus referee for useful suggestions regarding the exposition of the paper.

\section{Quasiconformal and quasisymmetric maps}


Given $a,b,c,d\in \bar{\mathbb{C}}=\mathbb{C}\cup\{\infty\}$ distinct, define the {\it cross-ratio}
$$
\mathrm{cr}(a,b,c,d)=\frac{(c-b)(d-a)}{(b-a)(d-c)}.
$$
Note that $\mathrm{cr}(1,i,-1,-i)=1$. In general,  $\mathrm{cr}(a,b,c,d)=1$ if and only if the hyperbolic geodesic in $\mathbb{D}$ with endpoints $a,c\in S^1$ is orthogonal to the hyperbolic geodesic with endpoints $b,d\in S^1$. The same statement is true when $a,b,c,d\in\bar{\mathbb{R}}=\mathbb{R}\cup\{\infty\}$ for the corresponding geodesics in the upper half-plane $\mathbb{H}:=\{ z=x+iy\in\mathbb{C}:y>0\}$ equipped with the hyperbolic metric $\rho (z)|dz|=\frac{|dz|}{y}$.   

An orientation preserving homeomorphism $h:S^1\to S^1$ is $M$-{\it quasisymmetric} if there exists $M\geq 1$ such that
\begin{equation}
\label{eq:qs-circle}
\frac{1}{M}\leq\mathrm{cr}(h(a),h(b),h(c),h(d))\leq M.
\end{equation}
for $a,b,c,d\in S^1$ given in a counterclockwise order such that $\mathrm{cr}(a,b,c,d)=1$ (see Tukia-Vaisala \cite{TV}). 
The smallest constant $M$ such that the above holds is called the {\it constant of quasisymmetry}.
In a completely analogous fashion one can define an $M$-quasisymmetric map of $\bar{\mathbb{R}}=\mathbb{R}\cup\{\infty\}$. If $\gamma :\mathbb{H}\to \mathbb{D}$ is a M\"obius map then an orientation preserving homeomorphism $h:S^1\to S^1$ is $M$-quasisymmetric if and only if $\gamma^{-1}\circ h\circ \gamma :\bar{\mathbb{R}}\to\bar{\mathbb{R}}$ is $M$-quasisymmetric.  Also, if $h:S^1\to S^1$ is $M$-quasisymmetric and $\gamma$ is a M\"obius map preserving $\mathbb{D}$ then $\gamma\circ h$ is also quasisymmetric.  All this follows by the invariance of the cross-ratio under M\"obius maps.

A map $h:S^1\to S^1$ is quasisymmetric if and only if it extends to a quasiconformal map of the unit disk $\mathbb{D}=\{ z:|z|<1\}$ (see \cite{Ahlfors}, \cite{BeurlingAhlfors} and \cite{DouadyEarle}). The analogous statement holds for $h:\bar{\mathbb{R}}\to\bar{\mathbb{R}}$. Moreover, if an orientation preserving homeomorphism $h:\bar{\mathbb{R}}\to\bar{\mathbb{R}}$ with $h(\infty )=\infty$ satisfies
\begin{equation}
\label{eq:qs-real_line}
\frac{1}{M}\leq \frac{h(x+t)-h(x)}{h(x)-h(x-t)}\leq M
\end{equation}
for all $x\in\mathbb{R}$ and $t>0$, then $h$ is $M'$-quasisymmetric for some $M'=M'(M)\geq 1$ (see \cite{BeurlingAhlfors}). Note that $\frac{h(x+t)-h(x)}{h(x)-h(x-t)}=\mathrm{cr}(h(x-t),h(x),h(x+t), \infty )$ so that (\ref{eq:qs-circle}) implies (\ref{eq:qs-real_line}).  By \cite{BeurlingAhlfors}, we have that (\ref{eq:qs-real_line}) implies (\ref{eq:qs-circle}) with a different constant (see \cite{BonahonSaric} for the discussion).

A family of quasisymmetric maps of $S^1$ is said to be {\it normal} if every sequence  contains a subsequence which uniformly on $S^1$ converges to a quasisymmetric map. By the corresponding results on the normal families of quasiconformal maps (see \cite{GL}) we obtain the following condition for being a normal family. Let $\{ h_n\}_n$ be a family of $M$-quasisymmetric maps of $S^1$. If there exist $\delta >0$ and triples $(a_n,b_n,c_n)$ converging to a triple $(a,b,c)$ of distinct point on $S^1$ such that $\min \{ |h_n(a_n)-h_n(b_n)|,|h_n(b_n)-h_n(c_n)|,|h_n(c_n)-h_n(a_n)|\}\geq \epsilon$ then $\{ h_n\}_n$ is a normal family (see \cite[\S 2.4, page 70]{LV}).

\section{Proof of necessity of condition (\ref{eq:qs_horocycle}) in Theorem \ref{thm:main}}

Assume that $h:S^1\to S^1$ is a quasisymmetric map and we need to prove that $s_h:\mathcal{F}\to\mathbb{R}$ satisfies (\ref{eq:qs_horocycle}).
Fix a tip $p\in S^1$ and a complementary  ideal triangle $\Delta_p$ of $\mathcal{F}$ with one ideal vertex $p$. Let $\gamma_p :\mathbb{H}\to \mathbb{D}$ be a M\"obius map such that $\gamma_p^{-1}(p)=\infty$ and $\gamma_p^{-1}(\Delta_p)$ is an ideal triangle with vertices $0$, $1$ and $\infty$.
 Let $\gamma_{h(p)}:\mathbb{H}\to\mathbb{D}$ be a M\"obius map such that the homeomorphism $\gamma_{h(p)}^{-1} \circ h\circ\gamma_p$ of the extended real axis $\bar{\mathbb{R}}=\mathbb{R}\cup\{\infty\}$ fixes $0$, $1$ and $\infty$. Since the constant of quasisymmetry of $\gamma_{h(p)}^{-1} \circ h\circ\gamma$ equals to the constant of quasisymmetry of $h$ for all choices of $p$ and $\gamma_p$,
the condition (\ref{eq:qs_horocycle}) follows from (\ref{eq:qs-real_line}) applied to $x-t,x,x+t\in\mathbb{Z}$ because a horocycle in $\mathbb{H}$ based at $\infty$ is a Euclidean line parallel to the real axis and the restriction of the hyperbolic metric to this horocycle is a Euclidean metric scaled by a constant (see \cite{Saric1}). 

\section{Proof of sufficiency of condition (\ref{eq:qs_horocycle}) in Theorem 1.1}

The proof of the sufficiency of the condition (\ref{eq:qs_horocycle}) is by the contradiction (see \cite{Saric1}).  Assume that $s:\mathcal{F}\to\mathbb{R}$ satisfies condition (\ref{eq:qs_horocycle}).  
Note that we are only given the shear function without a map of $S^1$. We first define a {\it developing map} $h_s$  on the complementary ideal triangles of $\mathcal{F}$ into $\mathbb{D}$. On the triangle $\Delta_0$ with vertices $1$, $i$ and $-1$, we set $h_s$ to be the identity. Let $\Delta$ be an arbitrary complementary ideal triangle of $\mathcal{F}$. We connect $\Delta_0$ to $\Delta$ by an oriented geodesic arc $l$ with the initial point in $\Delta_0$ and consider the set of edges $\{ f_1,f_2,\ldots ,f_k\}$ of $\mathcal{F}$ that intersect $l$ indexed by the order of their intersection points with $l$. Each $f_i$ divides the unit disk $\mathbb{D}$ into two hyperbolic half-planes and we orient $f_i$ such that the half-plane containing $\Delta_0$ is to the left of $f_i$. Let $T_{f_i}^{s(f_i)}$ be the hyperbolic translation with the translation length $|s(f_i)|$, the translation axis $f_i$ and the initial point of $f_i$ (for the chosen orientation) is repelling for $T_{f_i}^{s(f_i)}$ if $s(f_i)\geq 0$, otherwise it is attracting. We define 
$$h_s|_{\Delta}=T_{f_1}^{s(f_1)}\circ T_{f_2}^{s(f_2)}\circ\cdots \circ T_{f_k}^{s(f_k)}.$$ 
The developing map $h$ is defined on each complementary ideal triangle of $\mathcal{F}$ to be a M\"obius map and it is discontinuous on an edge $f\in\mathcal{F}$ if and only if $s(f)\neq 0$. The restriction of the developing map $h$ to the two triangles adjacent to $f$ differ by the pre-composition with $T_{f}^{s(f)}$. Since $T_{f}^{s(f)}$ fixes the endpoints of $f$, we conclude that $h_s$ is well-defined on the endpoints of each edge $f\in\mathcal{F}$. Thus $h_s$ extends by the continuity to the endpoints of the edges of $\mathcal{F}$ which is a dense subset of $S^1$. The developing map $h_s$ is preserving the cyclic order thus it is  injective.

\subsection{$h_s$ is a homeomorphism} 
 
 We  establish that $h_s$ extends to a homeomorphism of $S^1$. In order to do so, we need a preliminary lemma.

\begin{lemma}
\label{lem:union_covers_cont} Let $\{\Delta_n\}_{n\in\mathbb{N}}$ be the family of complementary triangles for the Farey tesselation $\mathcal{F}$ and $h_s$ a developing map for $s:\mathcal{F}\to\mathbb{R}$.  
If $\cup_{n\in\mathbb{N}}h_s(\Delta_n)=\mathbb{D}$ then the developing map $h_s$ extends by the continuity to a homeomorphism of $S^1$. 
\end{lemma}

\begin{proof}
Let $\{\Delta_n\}_{n\in\mathbb{N}}$ be the family of all (closed) complementary ideal triangles of $\mathcal{F}$. Then $\cup_{n\in\mathbb{N}}\Delta_n=\mathbb{D}$ and the set $X$ of ideal vertices of $\{\Delta_n\}_{n\in\mathbb{N}}$ is dense in $S^1$. 
We established that $h_s$ is defined on $X$ and it preserves the cyclic order.  Since  $\cup_{n\in\mathbb{N}}h_s(\Delta_n)=\mathbb{D}$ we have that $h_s(X)$ is dense in $S^1$. Indeed, if $h_s(X)$ were not dense in $S^1$ then an open arc $I$ of $S^1$ would contain no points of $h_s(X)$. Let $g$ be a hyperbolic geodesic with ideal endpoints equal to the endpoints of $I$. Then the hyperbolic half-plane with boundary geodesic $g$ that faces $I$ cannot intersect any $h_s(\Delta_n)$ by the convexity which contradicts $\cup_{n\in\mathbb{N}}h_s(\Delta_n)=\mathbb{D}$. Therefore $h_s(X)$ is dense in $S^1$.

Let $z\in S^1\setminus X$. Then $z$ is the intersection of the ideal boundary arcs of a nested sequence of hyperbolic half-planes $P_n$ bounded by the edges $f_n^z$ of $\mathcal{F}$. Since $h_s$ preserves cyclic order, it follows that the half-planes $h_s(P_n)$ with boundary geodesics $h_s(f_n^z)$ are nested. If $h_s(f_n^z)$ accumulate to a geodesic $g$ then by the cyclic order preserving property of $h_s$ the half-plane with boundary $g$ (which does not contain  $h_s(f_n^z)$) is disjoint from $\cup_{n\in\mathbb{N}}h_s(\Delta_n)$ which is a contradiction. Therefore the sequence $h_s(f_n^z)$ accumulates to a single point $w\in S^1$ and we define $h_s(z)=w$. Thus $h_s$ extends to a map of $S^1$ which is injective because $h_s$ preserves cyclic order and for any open arc $I$ of $S^1$  there an edge $f_I$ of $\mathcal{F}$ such that $I$ contains both endpoints of $f_I$. To see that the map $h_s$ is onto, let $w\in S^1$ be arbitrary.  Since $\{ h_s(\Delta_n)\}$ is an ideal triangulation of $\mathbb{D}$ we have that either $w$ is an ideal vertex of some triangle $h_s(\Delta_{n(w)})$ or it is an accumulation of distinct edges $\{ h_s(f_{n_k})\}_{k=1}^{\infty}$. In the former case we find a vertex $z$ of $\Delta_{n(w)}$ such that $h_s(z)=w$. In the later case, the sequence of distinct edges $\{ f_{n_k}\}_k$ is nested because $h_s^{-1}$ also preserves cyclic order. The limit $z$ of $\{ f_{n_k}\}_k$ satisfies $h_s(z)=w$. Therefore $h$ is onto as well.

Next we prove that $h_s:S^1\to S^1$ is continuous. 
Let $z\in S^1$ be a vertex of $\mathcal{F}$. Choose two edges $f_1$ and $f_2$ of $\mathcal{F}$ with a common endpoint $z$ such that the other endpoints $z_1$ and $z_2$ are separated by $z$. Then the arc $(z_1,z_2)\subset S^1$ is a neighborhood of $z$. By choosing a sequence of mutually disjoint pairs of edges $(f_1^n,f_2^n)$ as above, the corresponding arcs $(z_1^n,z_2^n)$ are a basis of the neighborhoods of $z$ because $\cup_{n\in\mathbb{N}}\Delta_n=\mathbb{D}$ implies that the Euclidean sizes of $f_1^n$ and $f_2^n$ go to zero. The same is true for their images under $h_s$ by $\cup_{n\in\mathbb{N}}h_s(\Delta_n)=\mathbb{D}$. Therefore $h_s$ is continuous at $z$. 

If $z\in S^1\setminus X$ then a basis of neighborhoods of $z$ consists of the arcs on the ideal boundary of hyperbolic half-planes determined by a sequence of edges of $\mathcal{F}$ that separate $z$ from $0\in\mathbb{D}$. The same is true for their image under $h_s$ and the point $h_s(z)$. Therefore $h_s$ is continuous at $z\in S^1\setminus X$. Thus $h_s:S^1\to S^1$ is continuous and since $S^1$ is compact and Hausdorff it follows that $h_s$ is a homeomorphism which realizes $s$.
\end{proof}
 
 We are ready to prove that the developing map is a homeomorphism. 
 
\begin{lemma}
\label{lem:homeo}
Let $s:\mathcal{F}\to\mathbb{R}$ be a shear function satisfying the condition from Theorem \ref{thm:main}. Then the developing map $h_s$ extends to a homeomorphism of $S^1$.
\end{lemma} 

\begin{proof}
By Lemma \ref{lem:union_covers_cont}, it is enough to prove  that $\cup_{n\in\mathbb{N}}h_s(\Delta_n)=\mathbb{D}$. Assume on the contrary that $\cup_{n\in\mathbb{N}}h_s(\Delta_n)\neq \mathbb{D}$. Since $h_s$ preserves cyclic order it also preserves separation of geodesics. It follows that any boundary component of $\cup_{n\in\mathbb{N}}h_s(\Delta_n)$ is not contained in the set. Thus a boundary component is accumulated by a sequence $\{ h_s(f_k)\}_{k=1}^{\infty}$, where $\{ f_k\}_{k=1}^{\infty}$ is a sequence of distinct edges of $\mathcal{F}$. There are two possibilities: either $\{ f_k\}_{k\geq k_0}$ share a common ideal endpoint or there exists an infinite sequence of consecutive triples $\{ (f_{k_i},f_{k_i+1},f_{k_i+2})\}_{i=1}^{\infty}$ such that the common endpoint of $(f_{k_i},f_{k_i+1})$ is different from the common endpoint of $(f_{k_i+1},f_{k_i+2})$.

\begin{figure}[h]
\leavevmode \SetLabels
\L(.3*.05) $1(a)$\\
\L(.66*.05) $1(b)$\\
\L(.26*.78) $\delta_1$\\
\L(.25*.6) $\delta_2$\\
\L(.32*.38) $\delta_3$\\
\L(.4*.7) $g$\\
\L(.7*.7) $s_i$\\
\L(.626*.2) $h_s(f_{k_i})$\\
\endSetLabels
\begin{center}
\AffixLabels{\centerline{\epsfig{file =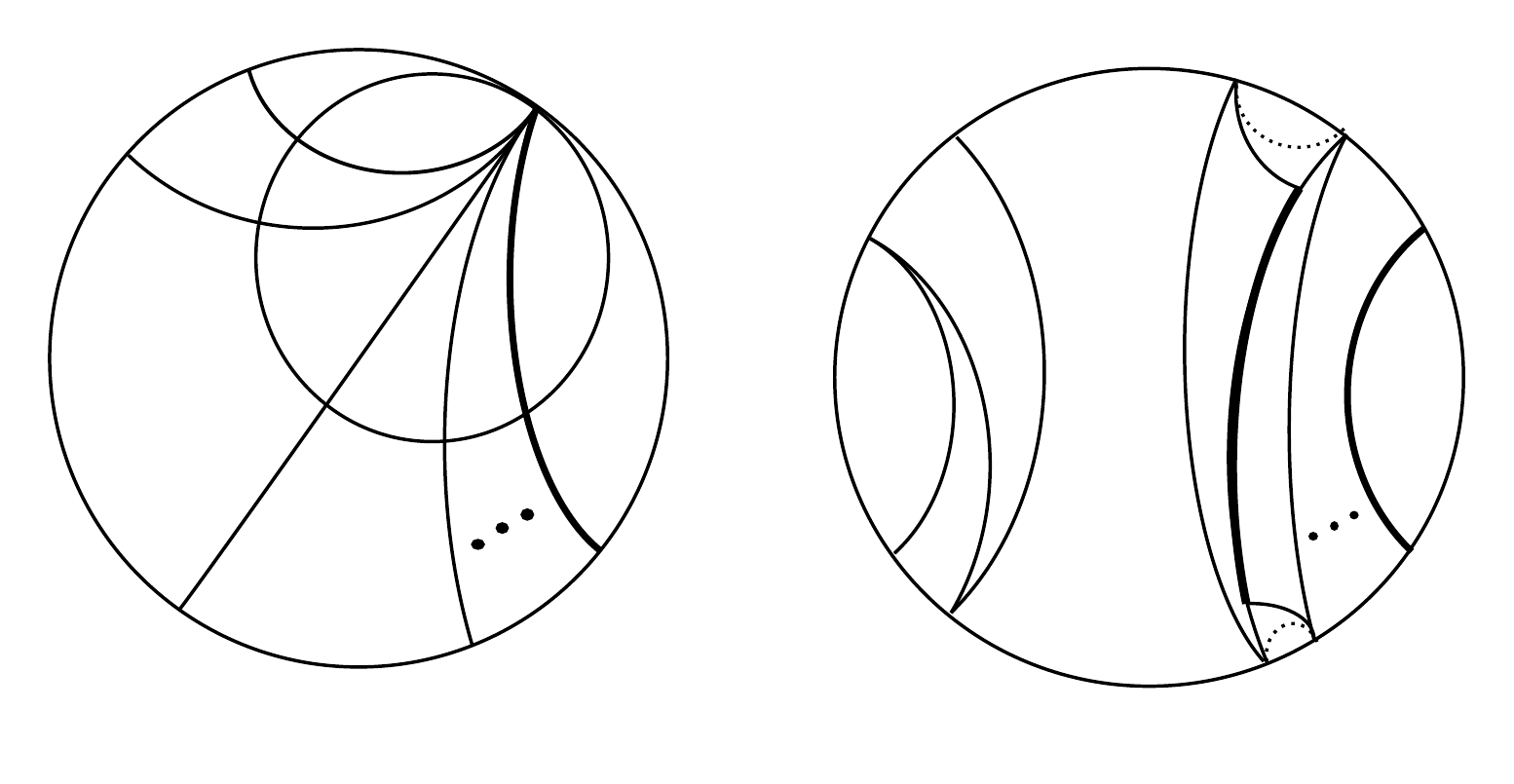,width=10.0cm,angle=0} }}
\vspace{-20pt}
\end{center}
\caption{The continuity of a developing map.}
\end{figure}
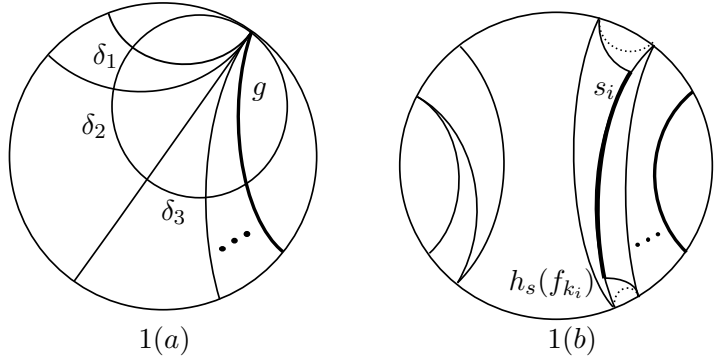

Assume we are in the former case-that is, there is a sequence $\{ f_k\}_{k=1}^{\infty}$ of consecutive edges of $\mathcal{F}$ in a tip $p$ such that $h_s(f_k)$ converges to a geodesic $g$. All geodesics $h_s(f_k)$ and $g$ share the same endpoint $h_s(p)$. We prove that (\ref{eq:qs_horocycle}) fails. Let $C$ be a horocycle based at $h_s(p)$ and denote by $\delta_k$ the length of the arcs of $C$ between $f_k$ and $f_{k+1}$. Since $h_s(f_k)\to g$ it follows that $\sum_{k=1}^{\infty}\delta_k<\infty$ (see Figure 1 (a)). This implies that $$\lim_{k\to\infty} \frac{\delta_1+\delta_2+\cdots +\delta_k}{\delta_{k+1}+\delta_{k+2}\cdots +\delta_{2k}}=\infty$$
which contradicts (\ref{eq:qs_horocycle}). 

In the later case, the triples $(h_s(f_{k_i}),h_s(f_{k_i+1}),h_s(f_{k_i+2}))$ converge to a fixed geodesic $g$ as $i\to\infty$ such that $g$ shares no common endpoint with the geodesics of the sequence. Let $s_i=|s(f_{k_i+1})|$.  Since the common endpoints of $(h_s(f_{k_i}),h_s(f_{k_i+1}))$ and $(h_s(f_{k_i+1}),h_s(f_{k_i+2}))$ are different, the Euclidean size of the third boundary geodesic of the complementary triangle of $h(\mathcal{F})$ with two sides $(h_s(f_{k_i}),h_s(f_{k_i+1}))$ goes to zero and the same is true for the third boundary geodesic of the triangle with sides $(h_s(f_{k_i+1}),h_s(f_{k_i+2}))$. The third boundary geodesics meet $h_s(f_{k_i+1})$ at opposite ideal endpoints (see Figure 1 (b)). This implies $\lim_{i\to\infty} s_i=\infty$ which again contradicts (\ref{eq:qs_horocycle}).
\end{proof}

\subsection{The sufficiency of condition (\ref{eq:qs_horocycle}) for a single fan}

Next we  prove Theorem \ref{thm:main} for the special case when the shear function is zero everywhere except on a single fan of geodesics. This is the main new ingredient (compared to \cite{Saric1}) in the proof of Theorem \ref{thm:main} for the general case and it makes the proof elementary. The proof  makes no use of the Douady-Earle extension unlike in the argument of \cite{Saric1} (or its duplicate in \cite{FanHu}).

\begin{lemma}
\label{lem:one_fan_qs}
Let $s:\mathcal{F}\to\mathbb{R}$ be a shear function that is equal to zero everywhere except on a single fan $\mathcal{F}^p=\{ f_n^p\}_{n\in\mathbb{Z}}$ with tip $p$. If there exists $M\geq 1$ such that for all $m\in\mathbb{Z}$ and $k\in\mathbb{N}\cup\{ 0\}$,
$$\frac{1}{M}\leq\frac{\delta_m^p+\delta_{m+1}^p+\cdots +\delta_{m+k}^p}{\delta_{m-1}^p+\delta_{m-2}^p+\cdots +\delta_{m-k-1}^p}\leq M
$$
then $s$ is induced by an $M'$-quasisymmetric map $h_s:S^1\to S^1$, where $M'$ depends only on $M$. 
\end{lemma}

\begin{proof}
By Lemma \ref{lem:homeo}, $s:\mathcal{F}\to\mathbb{R}$ is induced by a homeomorphism $h_s:S^1\to S^1$. We first conjugate $h_s$ by M\"obius maps such that it maps $\bar{\mathbb{R}}$ onto itself and the fan of geodesics $\mathcal{F}^p$ is replaced by the fan with tip at $\infty$.

Fix a complementary triangle $\Delta_p$ of $\mathcal{F}$ with one ideal vertex $p$. 
Let $\gamma_p:\mathbb{H}\to \mathbb{D}$ be a M\"obius map that sends $\infty$ to $p$ and an ideal triangle in $\mathbb{H}$ with vertices $0$, $1$ and $\infty$ onto the triangle $\Delta_p$. Then $\gamma_p^{-1}(\mathcal{F})$ 
is an ideal triangulation of $\mathbb{H}$. This ideal triangulation $\gamma_p^{-1} (\mathcal{F})$ is usually called the Farey tesselation of $\mathbb{H}$.
The endpoints of all edges of $\gamma_p^{-1} (\mathcal{F}^p)$ are precisely the integers $\mathbb{Z}$ with a common endpoint $\infty$. Let $\gamma :\mathbb{H}\to\mathbb{D}$ be the unique M\"obius map such that $h_1=\gamma^{-1}\circ h\circ \gamma_p$ fixes $0$, $1$ and $\infty$ on the extended real axis $\bar{\mathbb{R}}$.

The shear function $s$ pulls-back to a shear function $s\circ\gamma_p :\gamma_p^{-1}(\mathcal{F})\to \mathbb{R}$ for the 
 homeomorphism $h_1:\bar{\mathbb{R}}\to\bar{\mathbb{R}}$.
Notice that $s\circ\gamma_p$ is non-zero only on the edges of $\gamma_p^{-1}(\mathcal{F})$ that have an endpoint $\infty$. In order to simplify the notation, we replace $s\circ\gamma_p$ with $s$ for the rest of the proof.
 
To geometrically describe $h_1$, we take $h_1|_{[0,1]}=id$. An edge $f_n\in\gamma_p^{-1}(\mathcal{F})$ with endpoints $n$ and $\infty$ is given an orientation to the left as seen from the ideal hyperbolic triangle with vertices $0$, $1$ and $\infty$. 
 Given $a\in\mathbb{R}$ and $f_n\in\gamma_p^{-1}(\mathcal{F})$, let $T_{f_n}^a$ be a hyperbolic translation with the oriented axis $f_n$ and the signed translation length $a$.

If $x\in (n,n+1]$ for $n\in\mathbb{N}$, we have $h_1(x)=T_{f_1}^{s(f_1)}\circ T_{f_2}^{s(f_2)}\circ\cdots\circ T_{f_n}^{s(f_n)}(x)$ and if $x\in (-n,-n+1]$ for $n\in\mathbb{N}$,  we have $h_1(x)=T_{f_0}^{-s(f_0)}\circ T_{f_{-1}}^{-s(f_{-1})}\circ\cdots\circ T_{f_{-n}}^{-s(f_{-n+1})}(x)$. Then $h_1$ realizes $s$ on the edges of $\gamma_p^{-1}(\mathcal{F})$ and it is an increasing homeomorphism of $\bar{\mathbb{R}}$. Also note that the shear of $h_1$ on each edge $f$ of $\gamma_p^{-1}(\mathcal{F})$ which does not have an endpoint at $\infty$ is zero (see Figure 2 for the graph of $h_1$). 

\begin{figure}[h]
\leavevmode \SetLabels
\endSetLabels
\begin{center}
\AffixLabels{\centerline{\epsfig{file =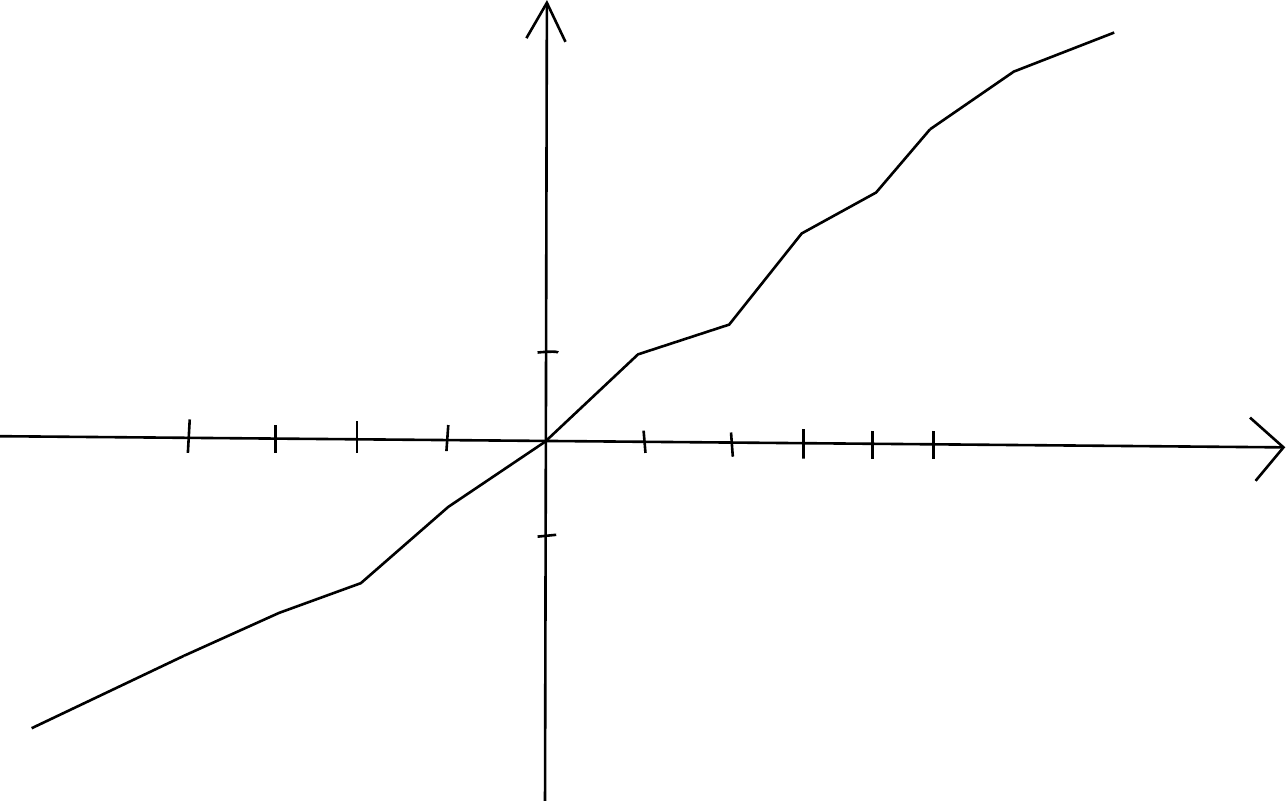,width=7.0cm,angle=0} }}
\vspace{-20pt}
\end{center}
\caption{The graph of the developing map $h_1$.}
\end{figure}

We first estimate $\frac{h_1(x+t)-h_1(x)}{h_1(x)-h_1(x-t)}$ for an arbitrary $x$ and $0<t\leq 5$. Note that $h_1$ is an increasing piecewise linear function on $\mathbb{R}$ with non-smooth points at $\mathbb{Z}$. The mean value theorem implies that $$h_1(x+t)-h_1(x)\leq [\max_{0\leq t_1\leq t}\frac{dh_1}{dx}(x+t_1)]t$$ and 
$$h_1(x)-h_1(x-t)\geq [\min_{0\leq t_1\leq t}\frac{dh_1}{dx}(x-t_1)]t.$$
Thus $$\frac{h_1(x+t)-h_1(x)}{h_1(x)-h_1(x-t)}\leq \frac{\max_{0\leq t_1\leq t}\frac{dh_1}{dx}(x+t_1)}{\min_{0\leq t_1\leq t}\frac{dh_1}{dx}(x-t_1)}\leq M^{10}$$
by the definition of $h_1$ and by $\max\{ \frac{\delta_{i+1}}{\delta_i},\frac{\delta_{i}}{\delta_{i+1}}\}=e^{|s(f_i)|}\leq M$ for each edge $f_i\in\gamma_p^{-1}(\mathcal{F})$.

In a similar fashion we obtain
$$\frac{h_1(x+t)-h_1(x)}{h_1(x)-h_1(x-t)}\geq \frac{\min_{0\leq t_1\leq t}\frac{dh_1}{dx}(x+t_1)}{\max_{0\leq t_1\leq t}\frac{dh_1}{dx}(x-t_1)}\geq M^{-10}$$

Assume now that $t>5$. To estimate $\frac{h_1(x+t)-h_1(x)}{h_1(x)-h_1(x-t)}$ from the above, let $a,b,c\in\mathbb{Z}$ be such that $[x-t,x]\supset [a,b]$, $[x,x+t]\subset [b,c]$ and $(b-a)+2=c-b$. 
The ratio $\frac{h_1(x+t)-h_1(x)}{h_1(x)-h_1(x-t)}$ is independent of  pre-composing $h_1$ by a translation $x\mapsto x-b$ since the translation preserves $\infty$. Thus we can assume that $b=0$.
Since $h_1(x)$ is increasing, we have
$$\frac{h_1(x+t)-h_1(x)}{h_1(x)-h_1(x-t)}\leq \frac{h_1(c)-h_1(b)}{h_1(b)-h_1(a)}.$$
We further normalize $h_1(x)$ by post-composing it with an affine map $x\mapsto Ax+B$ such that it equals the identity on $[0,1]$. Again $\frac{h_1(x+t)-h_1(x)}{h_1(x)-h_1(x-t)}$ is not affected.

In this normalization $h_1(x)$ is obtained by shearing by the amount $s$ with the identity on the initial complementary region between the vertical half-lines with the initial points $0$ and $1$ on $\mathbb{R}$. By a direct computation we obtain
\begin{equation}
\label{eq:h_1(c)}
\begin{split}
h_1(c)=e^{s(f_{c-1})+s(f_{c-2})+s(f_{c-3})+\cdots +s(f_1)}+e^{s(f_{c-2})+s(f_{c-3})+s(f_{c-4})+\cdots +s(1)}+\\
e^{s(f_{c-3})+s(f_{c-4})+\cdots +s(f_1)}+e^{s(f_{c-4})+s(f_{c-5})+\cdots +s(f_1)}+\cdots +
e^{s(f_1)}+1.
\end{split}
\end{equation}
Since
$$
h_1(c-2)=e^{s(f_{c-3})+s(f_{c-4})+\cdots +s(f_1)}+e^{s(f_{c-4})+s(f_{c-5})+\cdots +s(f_1)}+\cdots +
e^{s(f_1)}+1,
$$
equation (\ref{eq:h_1(c)}) gives
$$
h_1(c)\leq (e^{s(f_{c-1})+s(f_{c-2})}+e^{s(f_{c-2})}+1)h_1(c-2)\leq (M^2+M+1)h_1(c-2).
$$
Since $h_1(b)=0$ we have
$$
\frac{h_1(c)-h_1(b)}{h_1(b)-h_1(a)}=\frac{h_1(c)}{-h_1(a)}\leq \frac{h_1(c-2)}{-h_1(a)}(M^2+M+1).
$$
Since $c-2$ and $a$ are symmetric in $b=0$, the condition (\ref{eq:qs_horocycle}) implies $\frac{h_1(c-2)}{-h_1(a)}\leq M$ and we obtained
$$
\frac{h_1(x+t)-h_1(x)}{h_1(x)-h_1(x-t)}\leq (M^2+M+1)M\leq 3M^3.
$$

An analogous reasoning gives a lower bound of $1/(3M^3)$ on $\frac{h_1(x+t)-h_1(x)}{h_1(x)-h_1(x-t)}$ and we established that $h_1$ is $M':=\max\{ 3M^3,M^{10}\}$-quasisymmetric.
\end{proof}

\subsection{The sufficiency in the general case}

We complete the proof that if a general shear function $s$ satisfies the condition (\ref{eq:qs_horocycle}) then the induced homeomorphism $h_s:S^1\to S^1$ is quasisymmetric. Assume on the contrary that $h_s:S^1\to S^1$ is not quasisymmetric and we seek a contradiction. A homeomorphism $h_s$ is not quasisymmetric if and only if there exists a sequence of quadruples of points $\{ (a_n,b_n,c_n,d_n)\}_{n=1}^{\infty}$ on $S^1$ given in the counterclockwise order such that $\mathrm{cr} (a_n,b_n,c_n,d_n)= 1$ and $\mathrm{cr}(h_s(a_n),h_s(b_n),h_s(c_n),h_s(d_n))\to\infty$ as $n\to\infty$.

Let $A_n:\mathbb{D}\to\mathbb{D}$ be the M\"obius map such that $A_n:1,i,-1,-i\mapsto a_n,b_n,c_n,d_n$. Let $\mathcal{F}_n=A_n^{-1}(\mathcal{F})$ be an ideal triangulation of $\mathbb{D}$ which maps onto the Farey tesselation $\mathcal{F}$ by $A_n$. Then $h_s\circ A_n$ maps $\mathcal{F}_n$ onto $h_s(\mathcal{F})$. Our assumption is equivalent to 
\begin{equation}
\label{eq:not_qs}
\mathrm{cr}(h_s\circ A_n(1),h_s\circ A_n(i),h_s\circ A_n(-1),h_s\circ A_n(-i))\to\infty
\end{equation}
as $n\to\infty$.

The idea of the proof is to choose a sequence of M\"obius maps $B_n$ such that $B_n\circ h_s\circ A_n$ is a normal sequence of quasisymmetric maps. A limiting map $h^{*}$ of a subsequence of $B_n\circ h_s\circ A_n$ is a quasisymmetric homeomorphism of $S^1$. Thus it satisfies $\mathrm{cr}(h^{*}(1),h^{*}(i),h^{*}(-1),h^{*}(-i))<\infty$ which contradicts (\ref{eq:not_qs}). The choice of $B_n$ and the normality of $B_n\circ h_s\circ A_n$ is not straight forward and is the core of the proof. In summary, we prove a quasisymmetry of $h_s$ by showing that $B_n\circ h_s\circ A_n$ has a subsequence that pointwise converges to a homeomorphism. 

\begin{remark}
The fact that the shears can be both positive and negative is what makes direct estimations of the developing maps in terms of their shear functions a more challenging problem. In the case of earthquakes the measures are positive which allows an easier estimates of the developing maps. Thus it is not a mere convenience that we use an indirect proof.
A reader should think of conditionally convergent series of real numbers that is not absolutely convergent. 
\end{remark}

We divide the argument into two cases based on the limiting shape of $\mathcal{F}_n$ as follows. Let $\Delta_n^0$ be a complementary triangle of $\mathcal{F}_n$ which contains $0\in\mathbb{D}$. Denote by $x_n,y_n,z_n\in S^1$ vertices of $\Delta_n^0$. Since $0$ is in $\Delta_n^0$ it follows that it is not possible that all three vertices of $\Delta_n^0$ converge to a single point of $S^1$.  Therefore, there exists a subsequence of $\mathcal{F}_{n}$, which for the simplicity is denoted by $\mathcal{F}_n$ again, such that either $(x_{n},y_{n},z_{n})\to (x,y,z)$ for $x,y,z$ distinct, or $(x_{n},y_{n},z_{n})\to (x,x,z)$ for $x\neq z$ as $n\to\infty$. 

\subsubsection{The triangulations converge}
In the first case the geodesic triangles $\Delta_{n}^0$ converge to the triangle $\Delta^0_{*}$ with vertices $x,y,z$. The triangulation $\mathcal{F}_{n}$ is obtained by taking the image of $\Delta_{n}^0$ under the group generated by the hyperbolic reflections in the sides of $\Delta_{n}^0$. The triangle $\Delta_{n}^0$ is said to be the {\it base triangle}.
Since $\Delta_n^0$ converges to an ideal geodesic triangle $\Delta_*^0$, it follows that $\mathcal{F}_n$ converges to an ideal triangulation $\mathcal{F}_*$ with the base triangle $\Delta_*^0$. The edges of $\mathcal{F}_*$ are images of the edges of $\Delta_*^0$ under the group generated by the hyperbolic reflections in the edges of $\Delta_*^0$.

Recall that $A_n:\mathbb{D}\to\mathbb{D}$ is the M\"obius map such that $A_n:1,i,-1,-i\mapsto a_n,b_n,c_n,d_n$ and $\mathcal{F}_n=A_n^{-1}(\mathcal{F})$.
Let $A_n':\mathbb{D}\to\mathbb{D}$ be the M\"obius map which maps $(x,y,z)$ onto $(x_n,y_n,z_n)$. Then  $A_n'(\mathcal{F}_{*})=\mathcal{F}_n$ because $\mathcal{F}_n$ and $\mathcal{F}_{*}$ are both obtained by repeated hyperbolic reflections in the sides of $\Delta_n^0$ and $\Delta_{*}^0$. 
Define $(a_n',b_n',c_n',d_n'):=(A_n')^{-1}(1,i,-1,-i)$. Notice that $A_n'\to id$ and thus $(a_n',b_n',c_n',d_n')\to (1,i,-1,-i)$.
Let $B_n:\mathbb{D}\to\mathbb{D}$ be the M\"obius map such that $h_n:=B_n\circ h_s\circ A_n\circ A_n'$ fixes $(x,y,z)$. Then (\ref{eq:not_qs}) is equivalent to 
\begin{equation}
\label{eq:diverg_cross}
\mathrm{cr}(h_n(a_n'),h_n(b_n'),h_n(c_n'),h_n(d_n'))\to\infty
\end{equation}
as $n\to\infty$.

The shear function $s:\mathcal{F}\to\mathbb{R}$ pulls back to $s_n:=s\circ A_{n}\circ A_n':\mathcal{F}_{*}\to\mathbb{R}$ and it is the shear function of $h_n$. Since $e^{s_n(f_*)}$ is bounded between $1/M$ and $M$ for each edge $f_*\in\mathcal{F}_*$, it follows that we can choose a subsequence $s_{n_k}$ that converges on each $f_*\in\mathcal{F}_*$ to a shear function $s_*:\mathcal{F}_*\to\mathbb{R}$. Since $s_{n_k}:\mathcal{F}_*\to\mathbb{R}$ satisfies (\ref{eq:qs_horocycle}) with the same constant $M$, it follows that $s_*$ also satisfies (\ref{eq:qs_horocycle}) with the same constant $M$. By Lemma \ref{lem:homeo}, there is a homeomorphism $h_*:S^1\to S^1$ that realizes $s_*$.

Note that $h_{n_k}$ converges to $h_*$ pointwise  on the set $X_*$ of vertices of $\mathcal{F}_*$ by the definition of the developing maps for $s_{n_k}$ and $s_*$, and the convergence  $s_{n_k}(f_*)\to s_*(f_*)$ for each $f_*\in\mathcal{F}_*$. Since $X_*$ is dense in $S^1$ and $h_{n_k}, h_*$ are order preserving it follows that $h_{n_k}$ converges to $h_*$ pointwise on $S^1$.
Then
$$
\mathrm{cr}(h_{n_k}(a_{n_k}'),h_{n_k}(b_{n_k}'),h_{n_k}(c_{n_k}'),h_{n_k}(d_{n_k}'))\to \mathrm{cr}(h_{*}(1),h_{*}(i),h_{*}(-1),h_{*}(-i))<\infty
$$ 
as $k\to\infty$, and
we obtain a contradiction with (\ref{eq:diverg_cross}). Therefore the first case cannot occur.

\subsubsection{The triangulations degenerate}

Assume now that we are in the second case and seek a contradiction with (\ref{eq:not_qs}). A finite set $X\subset S^1$ is called an {\it $\epsilon$-net} in $S^1$ if every point of $S^1$ is on at most $\epsilon$ distance from a point in $X$.
In order to facilitate the proof, we need the following lemma. 

\begin{lemma}
\label{lem:limits}
Let $g_n,h_n:S^1\to S^1$ be two sequences of orientation preserving homeomorphisms such that $g_n$ converges uniformly to a homeomorphism $g_{\infty}:S^1\to S^1$. Assume that there exists a $\frac{1}{n}$-net $X_n$ in $S^1$ such that $h_n|_{X_n}=g_n|_{X_n}$. Then for all $x\in S^1$
$$
\lim_{n\to\infty} h_n(x)=g_{\infty}(x).
$$
\end{lemma}

\begin{proof}
Fix $\epsilon >0$ and $x\in S^1$. Let $a_n,b_n\in X_n$ be adjacent points such that $x\in (a_n,b_n)$, where $(a_n,b_n)$ is an arc of $S^1$ between $a_n$ and $b_n$ of length at most $\frac{1}{n}$. By the uniform continuity of $g_{\infty}$, there exists $n_0>0$ such that for $n\geq n_0$
\begin{equation}
\label{eq:ginf}
|g_{\infty}(a_n)-g_{\infty}(b_n)|<\epsilon .
\end{equation}
Let $a_n'=g_n(a_n)=h_n(a_n)$ and $b_n'=g_n(b_n)=h_n(b_n)$.  
Since $g_n$ converges to $g_{\infty}$ uniformly on $S^1$, we have that for $n\geq n_1>0$
$$
|g_{\infty}(a_n)-a_n'|<\epsilon\ \ \ \ \ \ \mathrm{and}\ \ \ \ \ \ |g_{\infty}(b_n)-b_n'|<\epsilon .
$$
These inequalities and (\ref{eq:ginf}) in turn imply that for $n\geq \max\{n_0,n_1\}$
\begin{equation}
\label{eq:gntog}
|a_n'-b_n'|<3\epsilon .
\end{equation}

Since $h_n(a_n)=a_n'$ and $h_n(b_n)=b_n'$ we get that both $h_n(x)$ and $g_n(x)$ belong to the arc $(a_n',b_n')=h_n((a_n,b_n))$. Then (\ref{eq:gntog}) gives for $n\geq \max\{n_0,n_1\}$
\begin{equation}
\label{eq:hngn}
|h_n(x)-g_n(x)|<|a_n'-b_n'|<3\epsilon .
\end{equation}

Finally, (\ref{eq:hngn}) implies for $n\geq \max\{n_0,n_1\}$
\begin{equation}
\label{eq:hng}
|h_n(x)-g_{\infty}(x)|\leq |h_n(x)-g_n(x)|+|g_n(x)-g_{\infty}(x)|<4\epsilon .
\end{equation}
Thus $\lim_{n\to\infty}h_n(x)=g_{\infty}(x)$ for all $x\in S^1$.
\end{proof}

Recall that $A_n:\mathbb{D}\to\mathbb{D}$ is the M\"obius map such that $A_n:1,i,-1,-i\mapsto a_n,b_n,c_n,d_n$ and $\mathcal{F}_n=A_n^{-1}(\mathcal{F})$. The complementary triangle $\Delta_n^0$ of $\mathcal{F}_n$ which contains $0$ has vertices $x_n,y_n,z_n$
such that $x_n,y_n\to x$ and $z_n\to z$ as $n\to\infty$, where $x\neq z$.  

Let $s_{z_n}:\mathcal{F}_n\to\mathbb{R}$ be the shear function that is equal to $s\circ A_n$ on the edges of $\mathcal{F}_n=A_n^{-1}(\mathcal{F})$ with one endpoint $z_n$ and is zero on all other edges. Let $h_{z_n}$ be the developing map for the shear function $s_{z_n}$ that is the identity on the triangle $\Delta_n^0$. Then there exists a M\"obius map $B_n:\mathbb{D}\to\mathbb{D}$ such that $h_{z_n}|_{X_n}=B_n\circ h_s\circ A_n|_{X_n}$. 

Fix $w\in S^1$ which is different from $x,z,x_n,y_n,z_n$. Then there exists a unique M\"obius map $B_n':\mathbb{D}\to\mathbb{D}$ such that $B_n'\circ h_{z_n}$ fixes $x_n$, $z_n$ and $w$. Define $g_n:= B_n'\circ h_{z_n}$ and $h_n:=B_n'\circ B_n\circ h_s\circ A_n$. By Lemma \ref{lem:one_fan_qs} the sequence $g_n$ is $M$-quasisymmetric for all $n$. Since the mutual distance between $x_n$, $z_n$ and $w$ has a positive lower bound, the sequence $g_n$ is normal.  Then there exists a subsequence of $g_n$ that converges to a quasisymmetric map $g_{\infty}$. For simplicity denote this subsequence by $g_n$ again. By the construction $h_n|_{X_n}=g_n|_{X_n}$. Then Lemma \ref{lem:limits} implies that $\lim_{n\to\infty}h_n(t)=g_{\infty}(t)$ for all $t\in S^1$. 
This is in a contradiction with (\ref{eq:not_qs}) as in the previous case and finishes the proof of Theorem 1.1.

\section{Correction to the proof in \cite{Saric1}}

 Lemma \ref{lem:one_fan_qs} is the main new ingredient in the proof of Theorem \ref{thm:main}. In \cite{Saric1}, we misinterpreted a lemma of Markovic \cite{Markovic} to claim that a control on the image under a homeomorphism $h$ of $S^1$ of four points on $S^1$ implies that the absolute value of the Beltrami coefficient of the Douady-Earle extension of $h$ is bounded away from $1$ near $0\in\mathbb{D}$. Fan and Hu \cite{FanHu} pointed out that  it is necessary to control eight points on $S^1$. They \cite{FanHu} proceeded with the exact steps of \cite{Saric1} to claim to have completed the proof. However, they did not prove a control on eight points of $S^1$(which is a major step when the triangulations degenerate) and they only tersely refer to \cite[Lemma 5.1]{Saric1} for the control. In \cite[Lemma 5.1]{Saric1}, we established the control on only four points. While it is possible to extend this lemma it is not a short computation. In this paper, we introduced Lemma \ref{lem:one_fan_qs} as a more geometric way of establishing this control and the control is on the whole $S^1$. Additional advantage of this approach is that the proof can be recast without the Douady-Earle extension \cite{DouadyEarle} which is given above.

For the completeness, we finish the proof in \cite{Saric1} using Lemma \ref{lem:one_fan_qs} and the Douady-Earle extension without refereeing to the lemma of Markovic. 
We replace $\mathbb{D}$ with the upper half-plane model $\mathbb{H}$ and the ideal boundary $S^1$ with the extended real axis $\bar{\mathbb{R}}=\mathbb{R}\cup\{\infty\}$. The Farey tesselation $\mathcal{F}$ of $\mathbb{H}$ is obtained by taking the base triangle $\Delta_0$ to have vertices $0$, $1$ and $\infty$, and other triangles to be the images of $\Delta_0$ under the group generated by the reflections in the sides of $\Delta_0$. 

In \cite{Saric1}, we assume that $s:\mathcal{F}\to\mathbb{R}$ satisfies (\ref{eq:qs_shears}) but the developing homeomorphism $h_s$ is not quasisymmetric. 
Since $h_s$ is not quasisymmetric it follows that the Doaudy-Earle extension $F_s$ of $h_s$ is not a quasiconformal map, i.e.-there exists a sequence $z_n\in\mathbb{H}$ such that $|Belt(F_s)(z_n)|\to 1$ as $n\to\infty$ where $Belt(F_s)(z):=\bar{\partial} F_s(z)/\partial F_s(z)$. The sequence $z_n$ leaves every compact subset of $\mathbb{H}$ because $F_s$ is real-analytic. Let $A_n\in PSL_2(\mathbb{Z})$ and $B_n \in PSL_2(\mathbb{R})$ such that $B_n\circ F_s\circ A_n$ fixes $0$, $1$ and $\infty$, and $z_n':=A_n^{-1}(z_n)$ is in the complementary triangle $\Delta_0$ of $\mathcal{F}$ with vertices $0$, $1$ and $\infty$. Since $A_n(\mathcal{F})=\mathcal{F}$, the shear function $s_n:=s\circ A_n:\mathcal{F}\to\mathbb{R}$ is well-defined. In addition, the map $B_n\circ h_s\circ A_n$ is a developing map of $s_n$. Note that the conformal naturality of the barycentric extension implies that 
$F_{s_n}=B_n\circ F_s\circ A_n$ and 
\begin{equation}
\label{eq:belt_1}|Belt(F_{s_n})(z_n')|\to 1
\end{equation} as $n\to\infty$. 

We seek a contradiction with (\ref{eq:belt_1}) and divide the argument into two cases based on the position of the sequence $z_n'\in\Delta_0$ (see \cite{Saric1}). The first case when a subsequence of $z_n'$ stays in a compact subset of $\Delta_0$ is unchanged.

The second case is where a correction is made. We assume that $z_n'$ leaves every compact subset of $\Delta_0$. After taking a subsequence and normalizing by a precomposition with an element of $PSL_2(\mathbb{Z})$ and a postcomposition by an element of $PSL_2(\mathbb{R})$, we can assume that $z_n'\to\infty$ inside $\Delta_0$ and $h$ fixes $0$, $1$ and $\infty$. The outline of the proof in \cite{Saric1} is as follows. Let $\lambda_n=Im(z_n')$ and $\lambda_n'\in\mathbb{R}$ such that $h_n'(x):=\frac{1}{\lambda_n'}h_{s_n}(\lambda_nx)$ fixes $0$, $1$ and $\infty$. 
One would like to show that $h_n'$ converges pointwise to a homeomorphism and then to finish the proof by the fact that a pointwise convergence of homeomorphisms implies pointwise convergence of the Beltrami coefficients of the corresponding Douady-Earle extensions.

Recall that $s_n:\mathcal{F}\to\mathbb{R}$ satisfies the condition (\ref{eq:qs_shears}) with a single $M$. We define $s_n^{\infty}(f)=s_n(f)$ for the edges $f\in\mathcal{F}$ with one endpoint at $\infty$ and $s_n^{\infty}(f)=0$ otherwise. Then Lemma \ref{lem:one_fan_qs} implies that $h_{s_n^{\infty}}$ is $M'$-quasisymmetric for each $n$. We define $f_n(x):=\frac{1}{\lambda_n'}h_{s_n^{\infty}}(\lambda_nx)$. Then $f_n$ are also $M'$-quasisymmetric for all $n$ and recall that $f_n$ fixes $0$, $1$ and $\infty$. The family $f_n$ is normal and thus there exists a subsequence $f_{n_k}$ that converges to an $M'$-quasisymmetric map $f_{\infty}$ of $\bar{\mathbb{R}}$. 

We will prove that $h_{n_k}'(x)=\frac{1}{\lambda_{n_k}'}h_{s_{n_k}}(\lambda_{n_k}x)$ converges uniformly on compact subsets $I\subset\mathbb{R}$ to $f_{\infty}$ which finishes the proof. Note that $\lambda_{n_k}\to\infty$ as $k\to\infty$. By definition  $h_{n_k}'$ and $f_{n_k}$ agree on the points $\frac{1}{\lambda_{n_k}}\mathbb{Z}$ of the real axis $\mathbb{R}$. We fix a compact interval $I\subset \mathbb{R}$ and prove uniform convergence on $I$.  

Let $x\in I$. Then there exist adjacent points $y_{n_k}^1,y_{n_k}^2\in\frac{1}{\lambda_{n_k}}\mathbb{Z}$ such that $x\in [y_{n_k}^1,y_{n_k}^2]$. Since both $f_{n_k}$ and $h_{n_k}'$ are increasing on $\mathbb{R}$ and agree at $\frac{1}{\lambda_{n_k}}\mathbb{Z}$ we have
$h_{n_k}'(x)\in [f_{n_k}(y_{n_k}^1),f_{n_k}(y_{n_k}^2)]$. Moreover, $f_{\infty}(x)\in [f_{\infty}(y_{n_k}^1), f_{\infty}(y_{n_k}^2)]$, the uniform convergence of $f_{n_k}$ to $f_{\infty}$ and $\lambda_{n_k}\to\infty$  imply that $h_{n_k}(x)$ is uniformly close to $f_{\infty}(x)$ for all $x\in I$ and $k$ large enough(for more details see the proof of Theorem \ref{thm:main}). Therefore we obtained the convergence of $h_{n_k}'$ to $f_{\infty}$ uniformly on $I$ and the proof is completed.

\end{document}